\documentclass[]{article}
\usepackage{amsmath,amsthm,amssymb,bm}
\usepackage[reset,a4paper,vmargin=3truecm,hmargin=2truecm]{geometry}
\usepackage{graphicx}
\graphicspath{{./pictures/}}
\usepackage{colortbl}
\include{pictures}

\numberwithin{equation}{section}

\theoremstyle{plain}
\newtheorem{thm}{Theorem}[section]
\newtheorem{lem}[thm]{Lemma}
\newtheorem{prop}[thm]{Proposition}
\newtheorem{cor}[thm]{Corollary}

\theoremstyle{definition}

\newtheorem{ex}{Example}[section]

\theoremstyle{remark}

\makeatletter
\def\mapstofill@{%
   \arrowfill@{\mapstochar\relbar}\relbar\rightarrow}
\newcommand*\xmapsto[2][]{%
   \ext@arrow 0395\mapstofill@{#1}{#2}}
\makeatother

\title{Expansion of Dirichlet L-function on the critical line in Meixner-Pollaczek polynomials}
\author{Hiroto Inoue\footnote{Research Fellow of Japan Society for the Promotion of Science.}}

\begin{document}
\maketitle
\begin{abstract}
We study the expansions of the completed Riemann zeta function and completed Dirichlet L-functions in Meixner-Pollaczek polynomials. We give the proof of the uniform convergence, the multiplicative structure for the coefficients of these expansions, and a calculation for the coefficients of a completed Dirichlet L-function $\widehat{L}(s, \chi_{-1})$. Furthermore, we give a boundary for the zeros of the approximating polynomial. 
\end{abstract}

\section{Introduction}
The Meixner-Pollaczek polynomials are defined by the hypergeometric expression
\[
q_n^{(\nu, \theta)}(s)
:=
e^{in\theta}\frac{(\nu)_n}{n!}
{}_2 F_1
\left( \begin{array}{c}
-n, \; s+\frac{\nu}{2}  \\
  \nu
\end{array};
2e^{-i\theta}\cos\theta
\right)
\quad
(n \in \mathbb{Z}, n\geq 0). 
\]
These polynomials $\{q_m^{(\nu, \theta)}(it)\}_{m\geq 0}$ form an orthogonal basis of the Hilbert space $L^2\left(\mathbb{R},  | \Gamma\left(it+\frac{\nu}{2}\right)|^2dt\right)$, see \cite{AAR1999}, \cite{FIW}. In this paper, we study the expansion of the completed Riemann zeta function $\Xi(t)$ in the  Meixner-Pollaczek polynomials $q_n^{(\frac32,0)}(it)$:
\begin{equation}\label{intro-MPex}
\Xi(t)=\sum_{n=0}^\infty a_n q_n^{(\frac32,0)}(it)
\end{equation}
in the Hilbert space $L^2\left(\mathbb{R},  | \Gamma\left(it+\frac34\right)|^2dt\right)$. 
Here the completed Riemann zeta function is defined by 
$
\Xi(t)=\pi^{-\frac{s}{2}}\Gamma\left(\frac{s}{2}\right)\zeta(s), 
\;
s=\frac12+2it,
$ 
where $\zeta(s)$ is the Riemann zeta function and $\Gamma(s)$ is gamma function. The same kind of expansion for $e^{\frac{\pi t}{2}}\Xi(t)$ is studied in \cite{Ku2007} and \cite{Ku2008}. There, the coefficients are given in terms of the Taylor coefficients of elementary functions. On the other hand, our expansion \eqref{intro-MPex} preserves the symmetry of the functional equation of $\Xi(t)$ by the parity of Meixner-Pollaczek polynomials, in other words, denoting the partial sums of \eqref{intro-MPex} by
\begin{equation}
S_n(t)
:=
\sum_{k=0}^{n} a_k q_k^{(\frac32, 0)}(it)
\quad
(n\in \mathbb{Z}, n\geq 0), 
\end{equation} 
we have $S_n(-t)=S_n(t)$. 
Therefore, we expect that \eqref{intro-MPex} is a good expansion of $\Xi(t)$. In fact, as the main theorem of this paper, we show that the polynomial $S_n(t)$ uniformly converges to $\Xi(t)$ as $n\rightarrow \infty$ in every compact set in the critical strip, and then each zero of $\Xi(t)$ is approximated by the zeros of $S_n(t) \; (n\geq 1)$; for $\rho\in \mathbb{C}$, 
\[
\Xi(\rho)=0
\Rightarrow
{}^{\exists}\left\{\rho_n \right\}_{n\geq 1} \text{ s.t. } S_n(\rho_n)=0,  
\;
\lim_{n\rightarrow \infty} \rho_n =\rho. 
\]
This fact suggests us a good approach to investigating the zeros of $\Xi(t)$. We have performed a numerical calculation for observing the location of zeros of $S_n(t)$ in $\mathbb{C}$. The result for $n=260$ is shown in Figure \ref{pseudo spectra 260}. 
\begin{figure}[h]
\centering
\includegraphics[scale=.25]{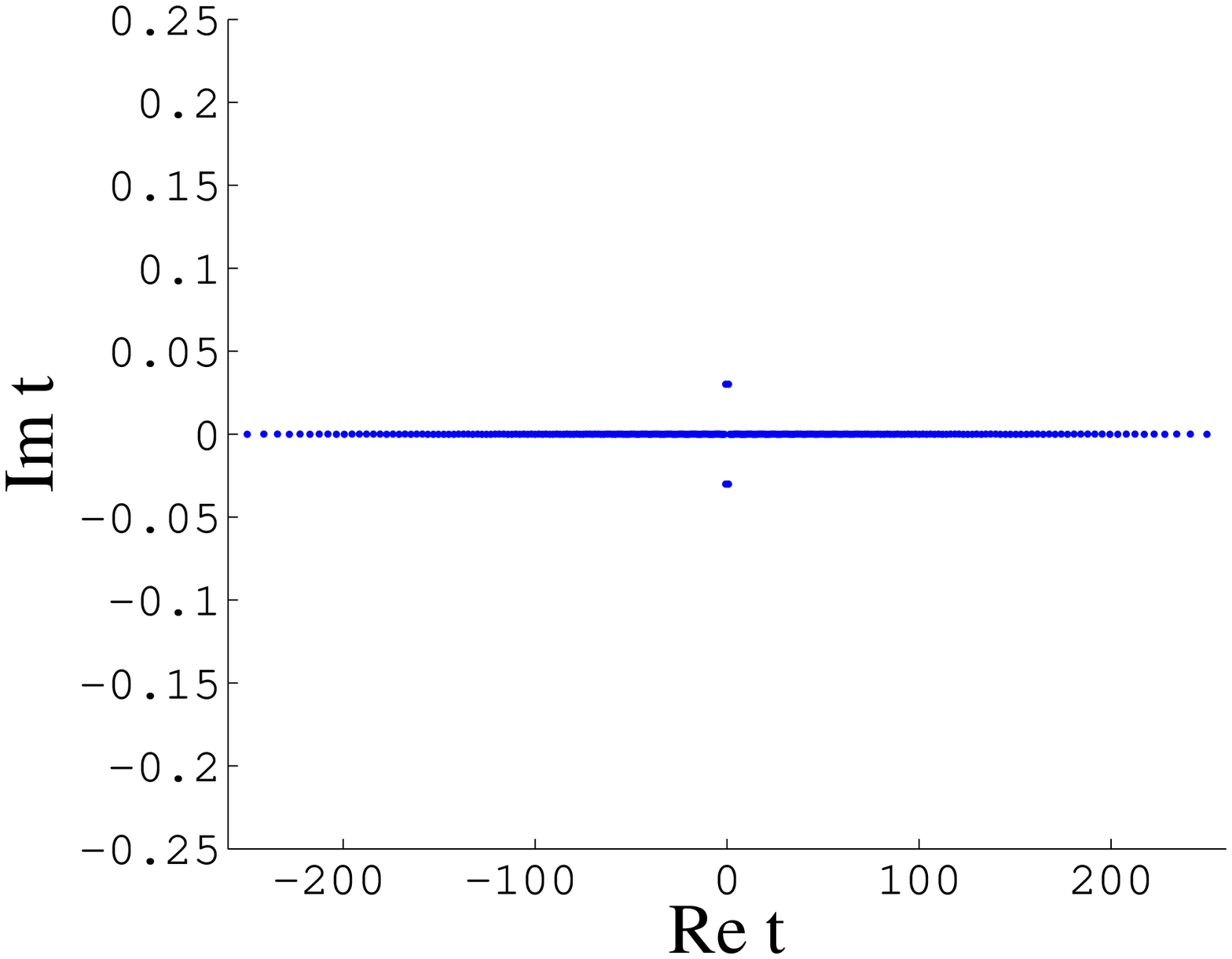}
\includegraphics[scale=.25]{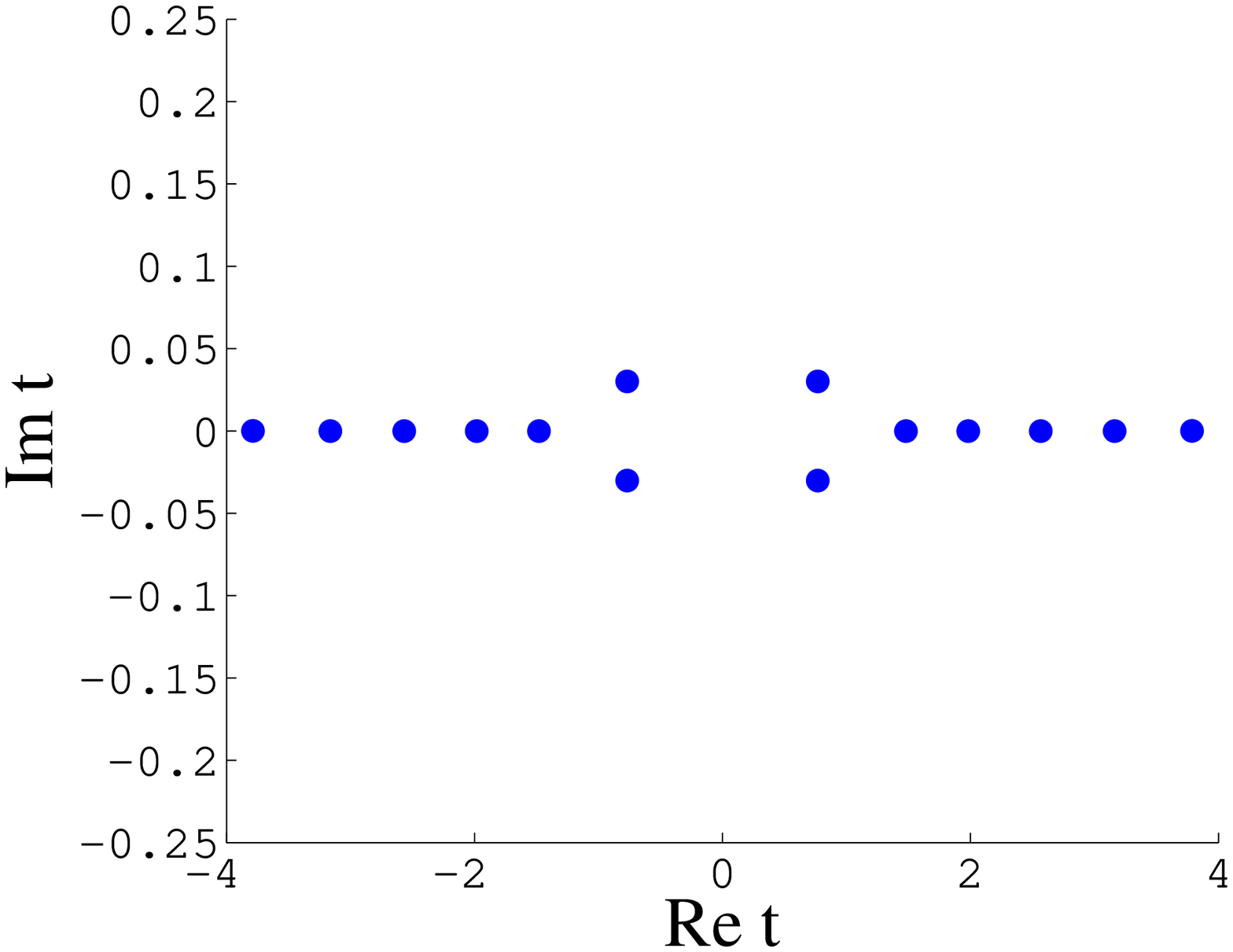}
\includegraphics[scale=.25]{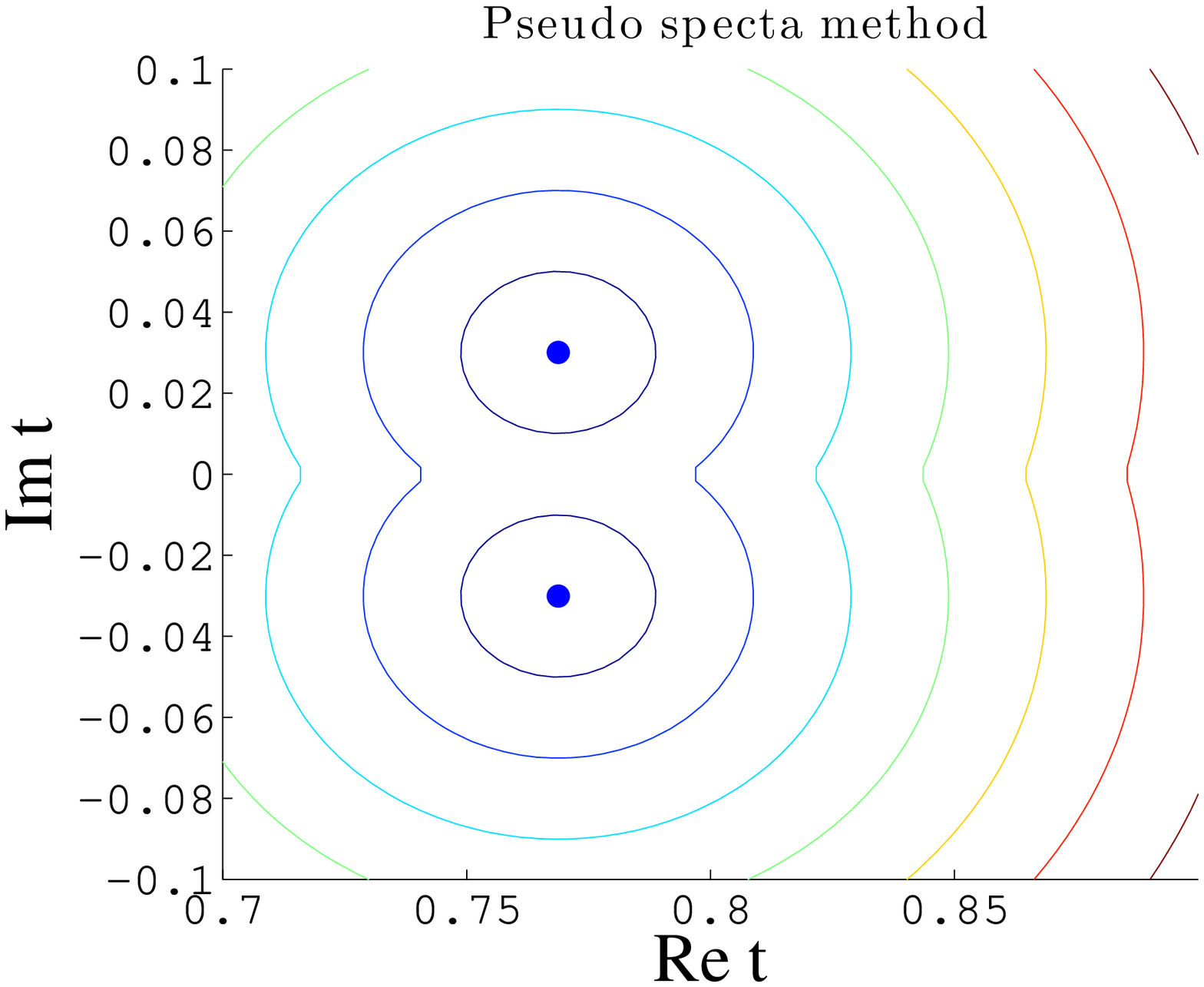}
\caption{left: Distribution of all the zeros of $S_{260}(t)$ in $\mathbb{C}$. center: The close-up at a neighbour of origin. right: The possible location of the complex zeros in $\Re t >0$ up to the estimation error, which is determined by the pseudo spectra method (see \cite{TE2005}). }
\label{pseudo spectra 260}
\end{figure}
We find that almost all the zeros of $S_n(t)$ lie in the real axis (left figure). But there is a set of four complex zeros in a neighbourhood of the origin (center figure). And these complex zeros likely exist according to the estimation error: The right figure shows the possible location of the complex zeros in $\Re t >0$ up to the estimation error. Motivated by this observation, we are interested in studying the expansion \eqref{intro-MPex} in detail and analyse the distribution of zeros of $S_n(t)$ because it would be a new approach for the Riemann hypothesis. 
\\
The paper is organized as follows. For a general function $f(it)\in L^2\left(\mathbb{R},  | \Gamma\left(it+\frac{\nu}{2}\right)|^2dt\right)$, let us call its expansion in $q_k^{(\nu, \theta)}(it)$ the MP-expansion of $f(it)$, and call its coefficients the MP-coefficients of $f(it)$. 
In Section \ref{Uniform convergence of MP-expansions}, we show the compact uniform convergence of MP-expansion for meromorphic functions with some conditions. 
In Section \ref{MP-coefficients}, we compute the MP-coefficients of a specific completed Dirichlet L-function using the results of  \cite{A2003} and \cite{Ku2007}. The result of such a computation is used for numerical calculations of the zeros of $S_n(t)$. 
In Section \ref{Approximating zeros}, as one of observations for distribution of the zeros of $S_n(t)$, we estimate a boundary for these zeros. 

\section{Uniform convergence of MP-expansions} 
  \label{Uniform convergence of MP-expansions}
Consider the MP-expansion of a function $f(it)$ with the parameters $\nu>0, \theta=0$;
\begin{equation}\label{MP-expansion}
f(it)=\sum_{n=0}^{\infty}a_n q_n^{(\nu,0)}(it)
\quad
 (t\in \mathbb{R}).
\end{equation}
Here $\{a_n\}_{n\geq 0}\subset \mathbb{C}$ are the MP-coefficients of $f(it)$. We argue the uniform convergence of this infinite sum. 
\subsection{Asymptotic formula of $q_n^{(\nu,0)}(it)$}
The asymptotic estimation $q_n^{(\nu,0)}(it)=O(n^{\frac{\nu}{2}-1 +|\Im t|})\quad (n\rightarrow \infty)$ is suggested in \cite{FKY1996}. Here $t$ is generally a complex number. The proof is due to the Darboux method, where the analytic property of the generating function is used to estimate the coefficients (see \cite{S1975}). 
\begin{lem}\label{estimate-MP} For fixed $t\in \mathbb{C}$ and sufficiently large $n$, a Meixner-Pollaczek polynomial $q_n^{(\nu,0)}(it)$ is estimated as
\[
q_n^{(\nu,0)}(it) = 2^{it-\frac{\nu}{2}}\frac{(-1)^n}{\Gamma(it+\frac{\nu}{2})}n^{it+\frac{\nu}{2}-1} +2^{-it-\frac{\nu}{2}}\frac{1}{\Gamma(-it+\frac{\nu}{2})}n^{-it+\frac{\nu}{2}-1}+O(n^{\frac{\nu}{2}-2+|\Im t|}).
\]
In particular,
\[
q_n^{(\nu,0)}(it)=O(n^{\frac{\nu}{2}-1 +|\Im t|})\quad (n\rightarrow \infty).
\]
\end{lem}
\begin{proof}
The generating function of $q_n^{(\nu, 0)}(it)$ is known as
\[
\sum_{n=0}^{\infty}q_n^{(\nu,0)}(it)w^n=(1-w)^{it-\frac{\nu}{2}}(1+w)^{-it-\frac{\nu}{2}}.
\]
We expand this function respectively in the vicinities of its singularities $w=-1$ and $w=1$ on the unit circle. In the vicinity of $w=-1$, we expand as
\begin{align*}
& 2^{ it-\frac{\nu}{2}} \sum_{k=0}^{\infty} \binom{ it-\frac{\nu}{2}}{k}\left(\frac{-1}{2}\right)^k (1+ w)^{k- it-\frac{\nu}{2}},
\\
\intertext{and similarly in the vicinity of $w=1$,}
&2^{- it-\frac{\nu}{2}} \sum_{k=0}^{\infty} \binom{- it-\frac{\nu}{2}}{k}\left(\frac{-1}{2}\right)^k (1- w)^{k+ it-\frac{\nu}{2}}.
\end{align*}
Fix $L\in \mathbb{N}, L>\frac{\nu}{2}$, and put $r_L^{(-1)}(w)$ and $r_L^{(1)}(w)$ to be the $L$th partial sums of these expansions respectively. Consider the difference
\[
H(w)=\sum_{n=0}^{\infty}q_n^{(\nu, 0)}(it)w^n-r_L^{(-1)}(w)-r_L^{(1)}(w).
\]
We see that the $L$th derivative of $H(w)$ is integrable on the unit circle $|w| = 1$. Thus if we expand $H(w)$ in the power series about $w=0$, the coefficients $d_n$ satisfy the condition
\[
d_n = O(n^{-L})
\]
according to the Cauchy's integral formula. Now $q_n^{(\nu, 0)}(it)$ is the sum of $d_n$ and the Taylor coefficients of $r_L^{(-1)}(w)$ and $r_L^{(1)}(w)$ at $w=0$:
\begin{align*}
q_n^{(\nu,0)}(it)=d_n&+2^{it-\frac{\nu}{2}}\sum_{k=0}^L\binom{it-\frac{\nu}{2}}{k}\left(\frac{-1}{2}\right)^k \binom{k-it-\frac{\nu}{2}}{n}
\\
&+2^{-it-\frac{\nu}{2}}\sum_{k=0}^L\binom{-it-\frac{\nu}{2}}{k}\left(\frac{-1}{2}\right)^k \binom{k+it-\frac{\nu}{2}}{n}(-1)^n.
\end{align*}
We can estimate each binomial coefficient by the Stirling formula as $\binom{s}{n}=\frac{(-1)^n}{\Gamma(-s)}\frac{\Gamma(-s+n)}{n!}\sim \frac{(-1)^n}{\Gamma(-s)}n^{-s-1}$, whence we find that the two terms of $k=0$ are  dominant for large $n$. Hence the result follows.
\end{proof}

\subsection{Asymptotic formula for $a_n$}
We denote a modified Mellin transform of a function $\psi \in L^2(\mathbb{R}_{> 0}, u^{\nu-1}du)$ by
\[
\mathcal{M}_{\nu}(\psi)(s)
:=
\frac{1}{\Gamma(s+\frac{\nu}{2})}
\int_0^{\infty} \psi(u) u^{s+\frac{\nu}{2}-1} du
\quad
(s\in i \mathbb{R}).
\]
The Plancherel formula of the Mellin transform is equivalent to following isomorphism between two Hilbert spaces
\[
L^2(\mathbb{R}_{> 0}, \frac{2^{\nu}}{\Gamma(\nu)}u^{\nu-1}du) \xrightarrow{\mathcal{M}_{\nu}}L^2(\mathbb{R},M_{\nu}(dt)),
\quad
M_{\nu}(dt):=\frac{1}{2\pi}\frac{2^{\nu}}{\Gamma(\nu)}|\Gamma(it+\frac{\nu}{2})|^2 dt. 
\]
The Meixner-Pollaczek polynomials $\{q_n^{(\nu, \theta)}(it)\}_{n\geq 0}$ are orthonormal basis of $L^2(\mathbb{R},M_{\nu}(dt))$, and each norm is $\int_{\mathbb{R}}|q_n^{(\nu, \theta)}(it)|^2M_{\nu}(dt)=\frac{(\nu)_n}{n!}$. Here $(\nu)_n:=\nu(\nu+1)\cdots (\nu+n-1)$ is the Pochhammer symbol. The image of $q_n^{(\nu, \theta)}(it)$ under the inverse modified Mellin transform $\mathcal{M}_{\nu}^{-1}$ is the Laguerre function $\psi_n^{(\nu, \theta)}(u):=e^{-u}L_n^{(\nu-1)}(2u)$. Recall the difference equation of the Meixner-Pollaczek polynomials
\[
D_{\nu}^{\mathrm{MP}} q_n^{(\nu, 0)}(s)=2n q_n^{(\nu, 0)}(s), 
\]
\[
D_{\nu }^{\mathrm{MP}} f(s):=\bigl(s+\frac{\nu}{2}\bigr)
\big\{f(s+1)-f(s)\big\} 
-\bigl(s-\frac{\nu}{2}\bigr)\big\{f(s-1)-f(s)\big\},
\]
and the differential equation of the Laguerre function
\[
D_{\nu}^{\mathrm{L}} \psi_n^{(\nu)}(u)=2n \psi_n^{(\nu)}(u), 
\]
\[
D_{\nu }^{\mathrm{L}}=-u\frac{d^2}{du^2}-\nu \frac{d}{du}+u-\nu. 
\]
We notice that they are equivalent by the relation $\mathcal{M}_{\nu}D_{\nu}^{\mathrm{L}}=D_{\nu}^{\mathrm{MP}}\mathcal{M}_{\nu}$. 
\begin{lem}\label{coefficients in MP-ex}
Take $\alpha, \beta, \delta \in \mathbb{R}$, and $s_-,s_+\in \mathbb{C}$ so that $\alpha+2<\Re(s_-)<0<\Re(s_+)<\beta-2$, $\delta>0$.
Let $f(s)$ be a meromorphic function on $D=\{s\in \mathbb{C}|\Re(s)\in(\alpha-\delta, \beta+\delta)\}$ satisfying $f(s)=O(e^{r|s|})\; (s\rightarrow \infty, s\in D)$ with $r<\frac{\pi}{2}$. Assume that $f(s)$ has only one or two poles at $s=s_+, s_-\in D$, and each of them is simple. Then the coefficients $a_n\; (n \geq 0)$ in \eqref{MP-expansion} have the asymptotic formula
\[
a_n\sim A_- n^{s_- -\frac{\nu}{2}}+(-1)^{n+1}A_+ n^{-s_+ -\frac{\nu}{2}}
\quad (n\rightarrow \infty),
\] 
where
$A_{-}= 2^{ s_{-}+\frac{\nu}{2}}\mathrm{Res}(f, s_-)\Gamma(-s_-+\frac{\nu}{2})$, and 
$A_{+}= 2^{ -s_{+}+\frac{\nu}{2}}\mathrm{Res}(f, s_+)\Gamma(s_++\frac{\nu}{2})$. In above formula, we take only $n\geq 0$ such that the right hand side dose not vanish. 
\end{lem}
\begin{figure}[h]
\centering
\includegraphics[scale=1.5]{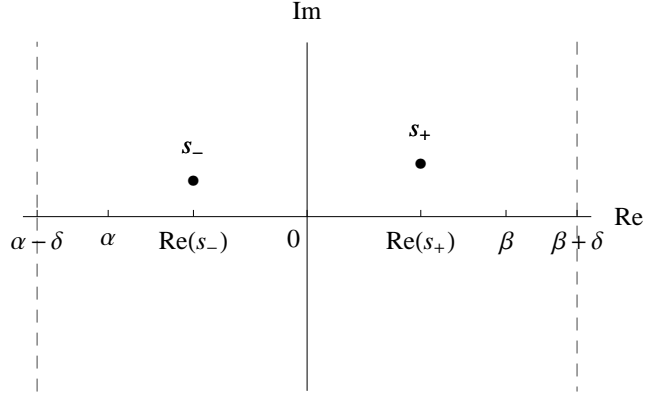}
\caption{The strip $D$, bounded by the two dashed lines. }
\end{figure}
\begin{proof}
Put 
\[
p_-(s):=\frac{\Gamma(s-s_-)}{\Gamma(s+\frac{\nu}{2})}\Gamma\left(s_-+\frac{\nu}{2}\right)\times\text{Res}(f,s_-),
\quad
p_+(s):=-\frac{\Gamma(-s+s_+)}{\Gamma(-s+\frac{\nu}{2})}\Gamma\left(-s_++\frac{\nu}{2}\right)\times\text{Res}(f,s_+)
\]
and
\[
g(s):=f(s)-(p_-(s)+p_+(s)).
\]
We claim that $(D_{\nu}^{\mathrm{MP}})^mg\in L^2(\mathbb{R},M_{\nu}(dt))$ for $m:=\min\{\lfloor\Re(s_+)\rfloor,\lfloor-\Re(s_-)\rfloor\}+1$. 
Consider the inverse of the modified Mellin transform
\begin{equation}\label{ghat}
\widehat{g}(u):=(\mathcal{M}_{\nu}^{-1}g)(u)=u^{-\frac{\nu}{2}}\frac{1}{2\pi i}\int_{\Re s=0}\Gamma\left(s+\frac{\nu}{2}\right)g(s)u^{-s}ds.
\end{equation}
By the asymptotic condition of $f(s)$, we see that $\widehat{g}\in L^2(\mathbb{R}_{>0}, u^{\nu-1}du) \cap C^{\infty}(\mathbb{R}_{>0})$. 
Each pole of the integrand in \eqref{ghat} on a region $D_-:=D\cap \{\Re(s)<0\}$ is contained in 
\[
\{-\frac{\nu}{2}-n\in D_- |\; n\in \mathbb{Z}, n \geq0\}
\cup
 \{s_--n\in D_-| \; n\in \mathbb{Z}, n\geq1\}.
\]
Thus by changing the path of the integral to $\Re(s)=\alpha$ and by Cauchy's theorem, $\widehat{g}(u)$ admits the asymptotic expansion at $u=+0$:
\begin{align}\label{asymptotic of ghat}
\widehat{g}(u)&=u^{-\frac{\nu}{2}}\left\{\sum_{\stackrel{n \geq0}{-\frac{\nu}{2}-n\in D_-}}c_-(n)u^{\frac{\nu}{2}+n}+\sum_{\stackrel{n\geq1}{s_--n\in D_-}}d_-(n)u^{n-s_-}+O(u^{-\alpha})\right\},
\end{align}
where each coefficient $c_-(n), d_-(n) \in \mathbb{C}$ is the corresponding residue of the integrand (see \cite{FGD1995}).
By the l'H\^{o}pital's rule, we have $k$-th derivative of \eqref{asymptotic of ghat} for $1\leq k \leq m$ as
\[
\left(\frac{d}{du}\right)^k\widehat{g}(u)
=
\sum_{\stackrel{n\geq k}{-\frac{\nu}{2}-n\in D_-}}c_- ^{(1)}(n)u^{n-k}+\sum_{\stackrel{n\geq1}{s_--n\in D_-}}d_- ^{(1)}(n)u^{-\frac{\nu}{2}+n-s_- -k}+O(u^{-\alpha-k-\frac{\nu}{2}})
\]
with $c_- ^{(1)}(n), d_- ^{(1)}(n) \in \mathbb{C}$. Therefore we have the asymptotic expansion of $(D_{\nu}^{\mathrm{L}})^m \widehat{g}(u)$ at $u=+0$ as
\[
(D_{\nu}^{\mathrm{L}})^m \widehat{g}(u)
=
\sum_{\stackrel{n\geq m}{-\frac{\nu}{2}-n\in D_-}}c_- ^{(2)}(n)u^{n-m}+\sum_{\stackrel{n\geq1}{s_--n\in D_-}}d_- ^{(2)}(n)u^{-\frac{\nu}{2}+n-s_- -m}+O(u^{-\alpha-m-\frac{\nu}{2}}).
\]
By the definition of $m$ and the fact that $\widehat{g}\in C^{\infty}(\mathbb{R}_{>0})$, we see that
\begin{equation}\label{behaviour at 0 of ghat}
\int_0^1 |(D_{\nu}^{\mathrm{L}})^m \widehat{g}(u)|^2 u^{\nu-1}du<\infty.
\end{equation}
On the other hand, each pole of the integrand in \eqref{ghat} on a region $D_+:=D\cap \{\Re(s)>0\}$ is contained in $\{s_+ +n\in D_+; n\in \mathbb{Z}, n\geq1\}$. 
Similarly, by changing the path of the integral to $\Re(s)=\beta$ (or by \cite{FGD1995} again), $\widehat{g}(u)$ admits the asymptotic expansion at $u=+\infty$:
\begin{equation*}
\widehat{g}(u)
=u^{-\frac{\nu}{2}}\left\{\sum_{\stackrel{n\geq1}{s_+ +n\in D_+}}d_+(n)u^{-n-s_+}+O(u^{-\beta})\right\},
\end{equation*}
where each coefficient $c_+(n), d_+(n) \in \mathbb{C}$ is the corresponding residue of the integrand. Then we can get the asymptotic expansion of $(D_{\nu}^{\mathrm{L}})^m \widehat{g}(u)$ at $u=+\infty$ in the same way as above, and see that 
\begin{equation}\label{behaviour at inf of ghat}
\int_1^{\infty} |(D_{\nu}^{\mathrm{L}})^m \widehat{g}(u)|^2 u^{\nu-1}du<\infty.
\end{equation}
Combining \eqref{behaviour at 0 of ghat} and \eqref{behaviour at inf of ghat}, we see that $(D_{\nu}^{\mathrm{L}})^m \widehat{g}(u)$ is contained in $L^2(\mathbb{R}_{>0}, u^{\nu-1}du)$. This proves our claim that $(D_{\nu}^{\mathrm{MP}})^m g(s)\in L^2(\mathbb{R},M_{\nu}(dt))$ by the equivalence $\mathcal{M}_{\nu}D_{\nu}^{\mathrm{L}}=D_{\nu}^{\mathrm{MP}}\mathcal{M}_{\nu}$. Now if we consider the expansion of $p_-(it)+p_+(it) \in L^2(\mathbb{R},M_{\nu}(dt))$
\begin{equation}\label{coeff a'n}
p_-(it)+p_+(it)=\sum_{n=0}a'_n q_n^{(\nu,0)}(it)
\quad (t\in \mathbb{R}),
\end{equation}
with $a'_n\in \mathbb{C}\;(n\geq0)$, then we see that
\[
\sum_{n=0}^{\infty}\frac{(\nu)_n}{n!}| (2n)^m (a_n-a'_n)|^2=\|(D_{\nu}^{\mathrm{MP}})^m g(s)\|^2 <\infty,
\]
where $\|\cdot\|$ is the norm defined by the measure $M_{\nu}(dt)$ on $\mathbb{R}$. Therefore $a_n-a'_n$ is estimated as
\begin{equation}\label{an-a'n}
a_n-a'_n=o(n^{-\frac{\nu-1}{2}-m})\quad (n\rightarrow \infty).
\end{equation}
For the rest, we calculate the coefficient $a'_n$ in \eqref{coeff a'n} by the orthogonality of the Meixner-Pollaczek polynomials $q_n^{(\nu,0)}(it)$ as
\begin{equation}\label{integral for a'n}
a'_n=\frac{n!}{(\nu)_n}
\int_{\mathbb{R}}\left(p_-(it)+p_+(it)\right)\overline{q_n^{(\nu,0)}(it)}M_{\nu}(dt),
\end{equation}
where $M_{\nu}(dt)=\frac{1}{2\pi}\frac{2^{\nu}}{\Gamma(\nu)}|\Gamma\left(it+\frac{\nu}{2}\right)|^2dt$. 
To calculate the first term of this integral, $\int_{\mathbb{R}}p_-(it)\overline{q_n^{(\nu,0)}(it)}M_{\nu}(dt)$, we use the expansion
\[
\overline{q_n^{(\nu,0)}(it)}=\frac{(\nu)_n}{n!}\sum_{k=0}^n \frac{(-n)_k }{(\nu)_k k!}\frac{\Gamma(-it+\frac{\nu}{2}+k)}{\Gamma(-it+\frac{\nu}{2})}2^k, 
\]
so that it is a linear combination of 
\begin{equation}\label{Gamma integral}
\int_{\mathbb{R}}\Gamma(it-s_-)\Gamma(-it+\frac{\nu}{2}+k)dt,
\quad
(0\leq k \leq n).
\end{equation}
Using an integral representation of hypergeometric function (Barnes's theorem, see \cite{AAR1999})
$$
\frac{\Gamma(a)\Gamma(b)}{\Gamma(c)}
{}_2F_1
\left(
\begin{array}{c}
a,\; b  \\
 c
\end{array}
; x\right)
 = \frac1{2\pi i} \int_{-i\infty}^{i\infty}\frac{\Gamma(a+s)\Gamma(b+s)\Gamma(-s)}{\Gamma(c+s)}(-x)^s \,ds
\quad 
(|\arg(-x)|<\pi)
$$
for $(a,b,c,x)=(-s_-+\frac{\nu}{2}+k,b,b,-1)\;(b>0)$, 
and the Pfaff transform
$$
{}_2F_1
\left(
\begin{array}{c}
a,\; b  \\
 c
\end{array}
; x\right)
= (1-x)^{-a} 
{}_2F_1
\left(
\begin{array}{c}
a,\; c-b  \\
 c
\end{array}
; \frac{x}{x-1}\right),
$$
we can calculate above integral \eqref{Gamma integral} and get
\begin{align}
\int_{\mathbb{R}}p_-(it)\overline{q_n^{(\nu,0)}(it)}M_{\nu}(dt)=&
\frac{2^{s_- +\frac{\nu}{2}}}{\Gamma(\nu)} 
\Gamma\left(s_-+\frac{\nu}{2}\right)\Gamma\left(-s_- +\frac{\nu}{2}\right)\text{Res}(f,s_-) \notag
\frac{(\nu)_n}{n!}
{}_2F_1
\left(
\begin{array}{c}
-n,\; -s_- +\frac{\nu}{2}  \\
\nu 
\end{array}
;1 \right)   \notag
\\
=&\frac{2^{s_- +\frac{\nu}{2}}}{\Gamma(\nu)}
\Gamma\left(-s_- +\frac{\nu}{2}\right)\text{Res}(f,s_-)
\frac{1}{n!}\Gamma\left(s_- +\frac{\nu}{2}+n\right). \label{integral of p-}
\end{align}
For the last equation we used following the Gauss summation formula (see \cite{AAR1999})
\[
\sum_{n=0}^{\infty}\frac{(a)_n (b)_n}{n!(c)_n}=
\frac{\Gamma(c)\Gamma(c-a-b)}{\Gamma(c-a)\Gamma(c-b)}, \quad \text{for } \Re(c-a-b)>0.
\]
For the other term of $p_+(it)$ in \eqref{integral for a'n}, we use the parity $\overline{q_n^{(\nu,0)}(it)}=q_n^{(\nu,0)}(-it)=(-1)^nq_n^{(\nu,0)}(it)$ and expand it into the Gamma functions to obtain
\begin{equation}\label{integral of p+}
\int_{\mathbb{R}}p_+(it)\overline{q_n^{(\nu,0)}(it)}M_{\nu}(dt)=
(-1)^{n+1}\frac{2^{\nu}}{\Gamma(\nu)} 2^{-s_+ -\frac{\nu}{2}}
\Gamma\left(s_+ +\frac{\nu}{2}\right)\text{Res}(f,s_+)
\frac{1}{n!}\Gamma\left(-s_+ +\frac{\nu}{2}+n\right),
\end{equation}
similarly as for $p_{-}(it)$.
By \eqref{integral for a'n}, \eqref{integral of p-} and \eqref{integral of p+} we have
\begin{align*}
a'_n=\frac{2^{\nu/2}}{\Gamma(\nu+n)}&\left\{2^{s_-}\text{Res}(f,s_-)\Gamma\left(-s_-+\frac{\nu}{2}\right)\Gamma\left(s_-+\frac{\nu}{2}+n\right)\right.
\\
&\left.+(-1)^{n+1}2^{-s_+}\text{Res}(f,s_+)
\Gamma\left(s_+ +\frac{\nu}{2}\right)
\Gamma\left(-s_+ +\frac{\nu}{2}+n\right)\right\}.
\end{align*}
Then we get the asymptotic formula $a'_n\sim A_- n^{s_- -\frac{\nu}{2}}+(-1)^{n+1}A_+ n^{-s_+ -\frac{\nu}{2}}\;(n\rightarrow \infty)$ by the Stirling formula. With the estimation \eqref{an-a'n} of $a_n-a'_n$, we get
\begin{align}\label{an estimate of an}
a_n=(a_n-a'_n)+a'_n
=
o(n^{-m-\frac{\nu-1}{2}})+O(n^{s_- -\frac{\nu}{2}}+n^{-s_+ -\frac{\nu}{2}})
=
O(n^{-l-\frac{\nu}{2}})
\quad (n\rightarrow \infty),
\end{align}
with $l=\min\left(\lfloor|\Re(s_{\pm})|\rfloor+\frac12,|\Re(s_{\pm})|\right)>0$. We obtained an estimate \eqref{an estimate of an} of the coefficients $a_n$ of given function $f(s)$, but now applying the same argument to the function $g(it)=\sum (a_n-a'_n)q_n^{(\nu,0)}(it)$, we get
\begin{equation}\label{an-an' again}
a_n-a'_n=O(n^{-l'-\frac{\nu}{2}})
\quad (n\rightarrow \infty),
\end{equation}
with $l'=\min\left(\lfloor|\Re(s_{\pm})|+1\rfloor+\frac12,|\Re(s_{\pm})|+1\right)=l+1$. By \eqref{an-an' again} and the asymptotic formula for $a'_n$, we get $a_n\sim A_- n^{s_- -\frac{\nu}{2}}+(-1)^{n+1}A_+n^{-s_+ -\frac{\nu}{2}}$ for $n\geq 0$ such that the right hand side dose not vanish. 
\end{proof}

Combining Lemma \ref{estimate-MP} and Lemma \ref{coefficients in MP-ex}, we have following estimation
\[
a_n q_n^{(\nu,0)}(it)=O(n^{-\min\{|\Re s_+|, |\Re s_-|\}+|\Im(t)|-1}).
\]
Then it follows that the MP-expansion of a meromorphic function with some conditions converges uniformly in every compact set in the strip sided by the nearest poles from the imaginary axis. Furthermore, we notice that the convergence is independent of $\nu$. 
\begin{thm}\label{uniform convergence of MP-ex}
Take $\alpha, \beta, \delta \in \mathbb{R}$, and $s_-,s_+\in \mathbb{C}$ so that $\alpha+2<\Re(s_-)<0<\Re(s_+)<\beta-2$, $\delta>0$.
Let $f(s)$ be a meromorphic function on $D=\{s\in \mathbb{C}|\Re(s)\in(\alpha-\delta, \beta+\delta)\}$ satisfying $f(s)=O(e^{r|s|})\; (s\rightarrow \infty, s\in D)$ with $r<\frac{\pi}{2}$. Assume that $f(s)$ has only one or two poles at $s=s_+, s_-\in D$, and each of them is simple. 
Then the MP-expansion \eqref{MP-expansion} converges to $f(it)$ uniformly in every compact set in the strip
\[
E(f):=\left\{t\in \mathbb{C};\; |\Im(t)|<\min\{|\Re s_+|, |\Re s_-|\}\right\}.
\]
\end{thm}
The limit of any sequence of polynomials converging uniformly is a holomorphic function in the converging region. On the other hand, we showed that the converse is also true when we consider the MP-expansion of a holomorphic function with some conditions in a strip. This fact gives us a natural reason to consider the MP-expansion of the completed Riemann zeta function and other $L$-functions, not only because of the parity of the Meixner-Pollaczek polynomials.

\section{MP-coefficients}
  \label{MP-coefficients}
We saw that an asymptotic formula for MP-coefficients is given by the location of poles of the expanded function. But it is needed to calculate MP-coefficients so explicitly that we can efficiently investigate zeros of approximating polynomial. In this section, we will see a multiplicative structure of MP-coefficients and calculate those of specific completed Dirichlet L-function. After that, we will remark a relation between MP-coefficients and the so-called deep Riemann hypothesis. 
\subsection{Multiplicative structure of MP-expansions}
%
We study the integral formula for triple product of $q_n^{(\nu, 0)}(it)$ to get the linearization of the product of two Meixer-Pollaczek polynomials. It has been obtained in \cite{A2003}. We just check the proof in our notation.  
Recall the generating function of the Meixner-Pollaczek polynomials $q_n^{(\nu, 0)}(it)$
\begin{equation}\label{gen-ft of MP}
\sum_{n=0}^{\infty} q_n^{(\nu, 0)}(iy) w^n
=
(1-w^2)^{-\frac{\nu}{2}} \left(\frac{1-w}{1+w}\right)^{iy} ,
\quad
|w|<1.
\end{equation}
\begin{prop}
When $l+m+n\equiv 0 \; (\mathrm{mod}\; 2)$ and $|n-l|\leq m \leq n+l$, 
\begin{equation}\label{integral of triple product for MP}
Q_{lmn}:=
\int_{ \mathbb{R}} 
q_l^{(\nu, 0)}(iy) q_m^{(\nu, 0)}(iy) \overline{q_n^{(\nu, 0)}(iy)} 
M_{\nu}(dy)
=
\frac{(-1)^{\frac{l+m+n}{2}} (\nu)_{\frac{l+m+n}{2}}}{\left(\frac{l+m-n}{2}\right)! \left(\frac{-l+m+n}{2}\right)! \left(\frac{l-m+n}{2}\right)!}. 
\end{equation}
Otherwise above integral is equal to $0$. 
\end{prop}
\begin{proof}
Put
\[
a=\left(\frac{1-r}{1+r}\right)\left(\frac{1-s}{1+s}\right)\left(\frac{1+t}{1-t}\right), 
\quad
b=(1+r)(1+s)(1-t). 
\]
For $s, r, t \in (-1, 1)$ in a neighbourhood of $0$, by the generating function \eqref{gen-ft of MP}, we have
\begin{align*}
\sum_{l, m, n=0}^{\infty}  Q_{lmn} r^l s^m t^n
&=
b^{-\nu}
\int_{\mathbb{R}}a^{iy-\frac{\nu}{2}} M_{\nu}(dy)
\\
&=
b^{-\nu}
\int_{\mathbb{R}}
\left(\frac{2}{1+a}\right)^{\nu}
\sum_{k=0}^{\infty} \left(\frac{1-a}{1+a}\right)^n q_k^{(\nu, 0)}(iy)
q_0^{(\nu, 0)}(iy)  M_{\nu}(dy).
\intertext{By the orthogonality of $q_k^{(\nu, 0)}$ and the definition of $a$ and $b$, we get}
&=
b^{-\nu} \left(\frac{2}{1+a}\right)^{\nu}
=
\left(\frac{1}{1+rs+st+tr}\right)^{\nu}.
\end{align*}
This function can be expanded in the Taylor series at $(r, s, t)=\mathbf{0}$ as
\[
\sum_{k=0}^{\infty} \binom{-\nu}{k} \sum_{\substack{a, b, c \geq 0 \\ a+b+c=k}} \frac{k!}{a! b! c!} (rs)^a (tr)^b (st)^c
=
\sum_{a, b, c=0}^{\infty} \binom{-\nu}{a+b+c} \frac{(a+b+c)!}{a! b! c!} r^{a+b} s^{a+c} t^{b+c}. 
\]
Comparing the coefficients $r^l s^m t^n$ between the first and the last formula, we get the result. 
\end{proof}

\begin{cor}\label{MP-ex of a product}
Let $Q_{lmn}\in \mathbb{R}$ be the coefficient defined in \eqref{integral of triple product for MP}. Take two functions $f, g \in L^2(\mathbb{R}, M_{\nu}(dt))$ with the MP-coefficients being respectively $a_n$ and $b_n$. If the product $(fg)(s)=f(s)g(s)$ is also contained in $L^2(\mathbb{R}, M_{\nu}(dt))$, then it is expanded as
\begin{equation*}
\left(\sum_{l=0}^{\infty}a_l q_l^{(\nu, 0)}(iy)\right)\left(\sum_{m=0}^{\infty}b_m q_m^{(\nu, 0)}(iy)\right)
=
\sum_{n=0}^{\infty} c_n q_n^{(\nu, 0)}(iy)
\end{equation*}
with the coefficients $c_n$ given by the following convolution
\begin{equation}\label{convolution of MP-coefficients}
c_n
=
\frac{n!}{(\nu)_n}\sum_{\substack{|l-m|\leq n \leq l+m\\ l+m+n \text{: even}}}a_l b_m Q_{lmn}
=
(-1)^n\sum_{j=0}^{\infty} \binom{-\nu-n}{j}
\sum_{k=0}^{n} \binom{n}{k}  a_{j+k} b_{j+n-k}.
\end{equation}
\end{cor}
\begin{proof}
For the last equation, change the indexes by $l=j+k$ and $m=j+n-k$. 
\end{proof}


\subsection{Generating function of MP-coefficients}
Here we see the integral formula for the generating function of MP-coefficients. We use the unitary isomorphism between the Hilbert spaces $\mathcal{M}_{\nu}: L^2(\mathbb{R}_{>0}, \frac{2^{\nu}}{\Gamma(\nu)} u^{\nu-1}du) \rightarrow L^2(\mathbb{R}, M_{\nu}(dt))$ by the modified Mellin transform $\mathcal{M}_{\nu}$. We use the formulae below (ref. \cite{AAR1999}, \cite{KLS}).
\begin{enumerate}

\item {Bessel function}
\begin{equation}\label{Bessel ft}
J_{\nu-1}(x)
=
\frac{(x/2)^{\nu-1}}{\Gamma(\nu)}
{}_0 F_1
\left(
\begin{array}{c}
 -  \\
 \nu
\end{array}
; \; -\left(\frac{x}{2}\right)^2  \right),
\quad
J_{1/2}(x)=\sqrt{\frac{2}{\pi x}}\sin x ,
\quad
J_{-1/2}(x)=\sqrt{\frac{2}{\pi x}}\cos x. 
\end{equation}

\item {Generating function of Laguerre polynomials} 
\begin{align}
&\sum_{n=0}^{\infty}\label{2nd gen-ft of Laguerre}
\frac{L_n^{(\nu-1)}(x)}{(\nu-1)_n}t^n
=
e^t
{}_0 F_1
\left(
\begin{array}{c}
 -  \\
 \nu
\end{array}
; \; -xt  \right)
=
e^t\frac{\Gamma(\nu)}{\sqrt{xt}^{\nu-1}}J_{\nu-1}\left(2\sqrt{xt}\right).
\end{align}
\end{enumerate}
\begin{lem}\label{lem gen-ft of MP-coeff}
For a function $f(it)\in L^2\left(\mathbb{R}, M_{\nu}(dt)\right)$, let $\{a_n\}$ be its MP-coefficients. And let $\varphi(u)\in L^2(\mathbb{R}_{>0}, u^{\nu-1}du)$ be the image of $f(it)$ under the inverse of the modified Mellin transform $\mathcal{M}_{\nu}^{-1}$: 
\begin{equation*}
\varphi(u)=
\sum_{n=0}^{\infty}a_n e^{-u}L_n^{(\nu-1)}(2u)
\xmapsto{\mathcal{M}_{\nu}}
f(it)=\sum_{n=0}^{\infty}a_n q_n^{(\nu, 0)}(it). 
\end{equation*}
Then the generating function of the MP-coefficients $a_n$ is given by
\begin{align}\label{gen-ft of MP-coeff}
&\sum_{n=0}^{\infty} \frac{(\nu)_n}{n!} \frac{a_n}{(\nu -1)_n}t^n
=
2^{\nu} e^t
\int_0^{\infty} \varphi(u) e^{-u} \frac{ J_{\nu-1}\left(2\sqrt{2ut}\right)}{\sqrt{2ut}^{\nu-1}}  u^{\nu-1}du
\end{align}
\end{lem}
\begin{proof}
We see that above integral is defined as the inner product of $L^2(\mathbb{R}_{>0}, u^{\nu-1}du)$. By the generating function \eqref{2nd gen-ft of Laguerre} and the orthogonality of the Laguerre polynomials, we get \eqref{gen-ft of MP-coeff}.
\end{proof}

\begin{ex}[Completed Riemann zeta function]
If $\varphi(u)=\sqrt{\frac{\pi}{u}}\left(\frac{1}{e^{2\sqrt{\pi u}}-1}-\frac{1}{2\sqrt{\pi u}}\right)$, then $f(it)=\mathcal{M}_{3/2}(\varphi)(it)$ is the completed Riemann zeta function $\Xi(t)$. 
The generating function of MP-coefficients $\{a_n\}$ in $\Xi(t)=\sum a_n q_n^{(\frac32, 0)}(it)$ is 
\begin{align*}
\sum_{n=0}^{\infty} \frac{(3/2)_n}{n!} \frac{a_n}{(1/2)_n}t^n
=
2^{\frac32} e^t 
\int_0^{\infty} 
\sqrt{\pi}\left(
\frac{1}{e^{2\sqrt{\pi u}}-1}-\frac{1}{2\sqrt{\pi u}}
\right)
 e^{-u} \frac{\sin(2\sqrt{2ut})}{\sqrt{2ut}} du.
\end{align*}
\end{ex}

\subsection{Completed Dirichlet $L$-functions}
Now we consider the Dirichlet $L$-function $L(s,\chi)=\sum_{n\geq 1}\chi (n)n^{-s}$ associated with a primitive character $\chi: \mathbb{Z}/N\mathbb{Z}\rightarrow \mathbb{C}^{\times}$. 
Define $\varepsilon(\chi)\in\{0, 1\}$ by $\chi(-1)=(-1)^{\varepsilon(\chi)}$. The completed $L$-function is defined by
\[
\widehat{L}(s,\chi)
:=
N^{\frac{s}{2}}\pi^{-\frac{s+\varepsilon(\chi)}{2}}
\Gamma\left(\frac{s+\varepsilon(\chi)}{2}\right)
L(s, \chi). 
\]
It satisfies the functional equation. In particular, we have $\widehat{L}(s, \chi)=\widehat{L}(1-s, \chi)$ if the character $\chi$ is real. 
%
%
\\
Recall the modified Mellin transform $\mathcal{M}_{\nu}(\varphi)(s)=\frac{1}{\Gamma(s+\nu/2)}\mathcal{M}(\varphi)\left(s+\frac{\nu}{2}\right)$, where $\mathcal{M}$ is the Mellin transform.
By the definition we have
\begin{equation}\label{mMellin integral rep. of L-function}
\widehat{L}(s, \chi)
=
\mathcal{M}_{\frac12+(1-\varepsilon)}
\left(\left(\frac{\pi}{u}\right)^{\frac{1-\varepsilon}{2}}
\sum_{n=1}^{\infty}\chi(n)e^{-2n\sqrt{\frac{\pi}{N}u}}\right)(it),
\quad
s=\frac12+2it,
\end{equation}
by changing the order of the integral and the summation, and by the duplication formula 
$\Gamma\left(\frac{s+\varepsilon}{2}\right)\Gamma\left(\frac{s+1-\varepsilon}{2}\right)=\sqrt{\pi}2^{1-s}\Gamma(s)$. By the periodicity of $\chi$, the integrand is simplified as
\begin{equation}\label{integrand for L-function}
\varphi(u)=
\left(\frac{\pi}{u}\right)^{\frac{1-\varepsilon}{2}}
\frac{1}{1-q^N}
\sum_{m=1}^{N}\chi(m) q^m,
\quad
q=e^{-2\sqrt{\frac{\pi}{N}u}}. 
\end{equation}
This rational formula of $q$ can be reduced and the cyclotomic polynomial $\Phi_N(q)$  appears in the denominator. 

\subsection{Computation of MP-coefficients for $\widehat{L}(s, \chi_{-1})$}
Here we compute the MP-coefficients of $\widehat{L}(s, \chi_{-1})$ associated with the character $\chi_{-1}:  \mathbb{Z}/4\mathbb{Z}\rightarrow \mathbb{C}^{\times}$ defined by $ \chi_{-1}(\bar{1})=1, \; \chi_{-1}(\bar{3})=-1$. 
By \eqref{mMellin integral rep. of L-function} and \eqref{integrand for L-function}, $\widehat{L}(\frac12 +2it, \chi_{-1})$ is the image of $\varphi(u)=\frac12 \mathrm{sech}(\sqrt{\pi u})$ by $\mathcal{M}_{1/2}$. 
Now considering an alternative function $\psi(u)=\frac12 \mathrm{sech}(\sqrt{\pi i u})$, we get the following relation: 
\\
\begin{center}
\begin{tabular}{ccc} 
$\displaystyle \varphi(u)=\frac12 \mathrm{sech}(\sqrt{\pi u})$ & $\xmapsto{\mathcal{M}_{1/2}}$ & $\displaystyle \widehat{L}(\frac12+2it, \chi_{-1})=\sum_{n=0}^{\infty}a_n q_n^{(\frac12, 0)}(it)$
\\ [0.5 em]
$\downarrow (u\rightarrow iu)$ &  & $\downarrow (\times i^{it-\frac{1}{4}})$
\\ [0.5 em]
$\displaystyle \psi(u)=\frac{1}{2} \mathrm{sech}(\sqrt{\pi i u})$ & $\xmapsto{\mathcal{M}_{1/2}}$ & $\displaystyle i^{it+\frac{1}{4}}\widehat{L}(\frac12+2it, \chi_{-1}) =\sum_{n=0}^{\infty}a'_n q_n^{(\frac12, 0)}(it)$
\end{tabular}
\end{center}
where $a_n$ and $a'_n$ are defined respectively to be the MP-coefficients of $\widehat{L}(\frac12+2it, \chi_{-1})$ and $i^{it+\frac{1}{4}}\widehat{L}(\frac12+2it, \chi_{-1})$. We here fix the $\log$ branch by $\log z=\log|z|+i\arg z\; (-\pi <\arg z < \pi)$. 
First we get a formula for $a'_n$, which is essentially the same one calculated in \cite{Ku2007}. 
And then, we get a formula for $a_n$ using the multiplicative structure of MP-coefficients. 
\begin{lem}
For $n\geq 0$, 
\begin{equation} \label{MP-coeff of exp chi-1 L-ft}
a'_n
=
\frac{1}{1-2n}F_n,
\end{equation}
where $F_n \; (n\geq 0)$ are recursively defined by
\begin{equation}\label{constants Fn}
\sum_{k=0}^n \frac{n! (2\pi i)^k}{(n-k)!(2k)!} F_{n-k}=\frac{(-1)^n\sqrt{i}-i}{\sqrt{2}}
\quad
(n\geq 0).  
\end{equation}
\end{lem}
\begin{proof}
By \eqref{Bessel ft} and Lemma \ref{lem gen-ft of MP-coeff}, we have
\begin{align}\label{gen-ft for chi-1 L-ft 2}
\sum_{n=0}^{\infty} \frac{(1/2)_n}{n!} \frac{a'_n}{(-1/2)_n} t^n 
&=
\sqrt{\frac{2}{\pi}} e^t
\int_0^{\infty} \frac{1}{2} \mathrm{sech}(\sqrt{\pi i u}) e^{-u} \cos(2\sqrt{2tu})  u^{-\frac12}du.
\end{align}
By variable change $u\rightarrow -i\pi u^2$ and the symmetry of integrand, this integral is identified with the Mordell integral
\begin{equation}\label{Mordell integral}
h(z)=h(z; \tau)
:=
\int_{\mathbb{R}}\frac{e^{\pi i \tau x^2 -2\pi zx}}{\cosh(\pi x)}dx
\quad
(\Im \tau \geq 0), 
\end{equation}
with $z=\sqrt{\frac{2it}{\pi}}, \; \tau=1$, up to constant. It is integrable (see \cite{Ku2007}), so that we have
\begin{equation*}
\sum_{n=0}^{\infty} \frac{(1/2)_n}{n!} \frac{a'_n}{(-1/2)_n} t^n \notag
=
\frac{\sqrt{i}e^{-t}-ie^{t}}{\sqrt{2}\cos(\sqrt{2\pi it})}. 
\end{equation*} 
The Taylor coefficients of this analytic function at $t=0$ are given by $F_n$ defined in \eqref{constants Fn}, and we get the result. 
\end{proof}
The generating function \eqref{gen-ft for chi-1 L-ft 2} of $a'_n$ is an integrable Mordell integral. On the other hand, it is not integrable for $a_n$. But we can get a formula for $a_n$ by $a'_n$. 
\begin{prop}
The MP-coefficients of the completed L-function associated with $\chi_{-1}$ in 
\begin{equation*}
\widehat{L}\left(\frac{1}{2}+2it, \chi_{-1}\right)=\sum_{n=0}^{\infty} c_n q_n^{(\frac12, 0)}(it)
\end{equation*}
are written as 
\begin{align}\label{formula of MP-coeff for cL(s, chi-1)}
c_n
=
(-2i)^{\frac{1}{4}}
\sum_{j=0}^{\infty} \binom{-n-1/2}{j}
\sum_{k=0}^n
 \binom{n}{k} 
 \frac{ i^{n-j+k} F_{j+k}}{1-2(j+k)}, 
\end{align}
and $c_n=0$ for odd $n$.  
\end{prop}
\begin{proof}
By the generating function of the Meixner-Pollaczek polynomials \eqref{gen-ft of MP}, we have
\begin{equation}\label{MP-ex of exp}
i^{-\frac14} \left(\frac{1-ir}{1+ir}\right)^{it}
=
 \sum_{n=0}^{\infty} \left(-i(1+r^2)\right)^{\frac14} (-ir)^n q_n^{(\frac12, 0)}(it) 
\end{equation}
for $0<r<1$. 
Then applying Corollary \ref{MP-ex of a product} with 
$
a_n=a'_n
=
\frac{1}{1-2n}F_n
$ in \eqref{MP-coeff of exp chi-1 L-ft} 
and 
$b_n=\left(-i(1+r^2)\right)^{\frac14} (-ir)^n$, we get the MP-coefficients $c_n(r)$ defined in
\begin{equation}\label{c_n^r MP-coeff}
\left(i\frac{1-ir}{1+ir}\right)^{it}
\widehat{L}\left(\frac12+2it, \chi_{-1}\right)
=
\sum_{n=0}^{\infty} c_n(r) q_n^{(\frac12, 0)}(it),
\end{equation}
as the convolution of $a_n$ and $b_n$. For the limitation $r\rightarrow 1$, both side of \eqref{c_n^r MP-coeff} converges in $L^2(\mathbb{R}, M_{\nu}(dt))$. Then we get \eqref{formula of MP-coeff for cL(s, chi-1)} by $c_n=\lim_{r\rightarrow 1-0}c_n(r)$. 
By the parity $q_n^{(\nu, 0)}(-s)=(-1)^n q_n^{(\nu, 0)}(s)$ and the functional equation $\widehat{L}(s, \chi_{-1})=\widehat{L}(1-s, \chi_{-1})$, it follows that $c_n=0$ for odd $n$.  
\end{proof}

\section{Approximating zeros}
  \label{Approximating zeros}
\subsection{Convergence of zeros of partial sums}

Theorem \ref{uniform convergence of MP-ex} ensures that each zero of the expanded function $f(it)$ is approximated by the zeros of the partial sums. Denote the $n$-th partial sum of MP-expansion for a function $f$ by
\[
S_n(f; t):=\sum_{k=0}^n a_k(f) q_k^{(\nu, 0)}(it). 
\]
When there is no risk of ambiguity, we just write $S_n(f; t)=S_n(t), \; a_n(f)=a_n$. We put $Z_n(f)$ to be the set of all the zeros of $S_n(f; t)$
\[
Z_n(f):=\left\{\rho \in\mathbb{C}|S_n(f; \rho)=0\right\}
\]
 for each $n\geq0$. 
\begin{prop}
Retain the same assumption and notations in Lemma \ref{coefficients in MP-ex}, Theorem \ref{uniform convergence of MP-ex}. 
Then for every zero $i\rho\in E(f)$ of $f(it)$, and for an arbitrary sufficient small $\varepsilon>0$, there exists $N\in\mathbb{N}$ such that
\[
n\geq N \Rightarrow \exists\rho_n\in Z_n(f) ; |\rho_n-\rho|<\varepsilon.
\]
In particular, if $i\rho$ is a simple zero, such a $\rho_n \in Z_n(f)$ is unique for each $n\geq N$.
\end{prop}
\begin{proof}
Take a zero $i\rho$ of $f(it)$, and an arbitrary small $\varepsilon>0$. We can assume that the disk $D_{\varepsilon}=\left\{t\in\mathbb{C}:|t-\rho|\leq\varepsilon\right\}$ has no other zeros of $f(it)$, and $D_{\varepsilon}\subset \left\{t\in \mathbb{C}:-\frac14 < \Im t < \frac14 \right\}$. Since $S_n(f; t)$ converges to $f(it)$ uniformly in $D_{\varepsilon}$ by Theorem \ref{uniform convergence of MP-ex}, we can take $N\in\mathbb{N}$ such that
\[
n\geq N \Rightarrow |f(it)-S_n(f; t)|<|f(it)| \quad(t\in \partial D_{\varepsilon}).
\]
Then by Rouch\'{e}'s theorem in complex analysis, $f(it)$ and $S_n(f; t)$ have the same number of zeros in $D_{\varepsilon}$ taking the multiplicity into account. So we get one of those zeros $\rho_n \in Z_n(f)$ as in the statement. If $i\rho$ is simple, $\rho_n$ is the unique simple zero.
\end{proof}
Therefore, we are interested in the distribution of zeros of $S_n(f; t)$ for large $n\geq0$, because it provides some information of the zeros of $f(it)$. For example, consider the MP-expansion of the Riemann zeta function $f(it)=\zeta(\frac12+it)$. If all zeros of $S_n(f; t)$ were real for infinitely many $n\geq 0$, then the Riemann hypothesis would be satisfied (but not vice versa). And it might suggest information about the multiplicity of each zero of $\zeta(\frac12+it)$.

\subsection{Determinant formula for partial sums}
The partial sum $S_n(f; t)$ is a linear combination of orthogonal polynomials. Such a polynomial is expressed as the characteristic polynomial of the so-called companion matrix. It means that the zeros of $S_n(f; t)$ coincide with the eigenvalues of this matrix. Such expressions for $S_n(f; t)$ and its zeros are also mentioned in \cite{Ku2008}. Most of all, in general, the numerical calculation of matrix eigenvalues are much faster than that of zeros of a polynomial. 
\\
To identify the companion matrix in our situation, we recall the recurrence formula of $q_n^{(\nu, 0)}(it)$
\begin{equation}\label{rec-form of q_n}
2it q_n^{(\nu, 0)}(it)=\left(n+\nu-1 \right)q_{n-1}^{(\nu, 0)}(it)-(n+1)q_{n+1}^{(\nu, 0)}(it).
\end{equation}
We use the notation
\[
\mathbf{f}_n(t):={}^t\left(q_0^{(\nu, 0)}(it),\ldots,q_{n-1}^{(\nu, 0)}(it)\right)
\]
for a column vector-valued function, and
\[
\mathbf{H}_n:=
\left(
\begin{array}{ccccc}
0 & -1 & & &  \\
\nu & 0 & -2 & &  \\
 & \ddots & \ddots & \ddots &  \\
 & & n+\nu -3 & 0 & -(n-1)  \\
 & & & n +\nu-2 & 0
\end{array}
\right)
, \quad
\mathbf{e}_n=
\left(
\begin{array}{c}
0  \\
\vdots  \\
0  \\
1 
\end{array}
\right),
\]
for an $n\times n$ tridiagonal matrix, and the $n$-th standard vector in $\mathbb{R}^n$. \eqref{rec-form of q_n} is rewritten as follows.
\begin{equation}\label{rec-form2 of q_n}
2it\mathbf{f}_n(t)=\mathbf{H}_n\mathbf{f}_n(t)-nq_n^{(\nu, 0)}(it) \mathbf{e}_n.
\end{equation}
In this manner, we also rewrite $S_n(f; t)=\sum_{k=0}^n a_k q_k^{(\nu, 0)}(it)$, for $n\geq 0$ such that $a_n\neq 0$, as
\begin{equation}\label{S_n2}
S_n(f; t)=\left(a_0,\ldots,a_{n-1}\right)\mathbf{f}_n(t)+a_n q_n^{(\nu, 0)}(it).
\end{equation}
Then combining the two formula \eqref{rec-form2 of q_n}, \eqref{S_n2} to eliminate the term of $q_n^{(\nu, 0)}(it)$, and putting
\begin{equation}\label{companion matrix}
\mathbf{B}_n:=\mathbf{H}_n+\frac{n}{a_n}\mathbf{e}_n\left(a_0,\ldots,a_{n-1}\right),
\end{equation}
we get
\begin{equation}\label{eigen relation}
2it\mathbf{f}_n(t)=\mathbf{B}_n\mathbf{f}_n(t)-\frac{n}{a_n}S_n(f; t)\mathbf{e}_n.
\end{equation}
The matrix $\mathbf{B}_n$ is called the companion matrix of $S_n(f; t)$ in term of $q_n^{(\nu, 0)}(it)$.
We easily see from \eqref{eigen relation} that for every zero $\rho$ of $S_n(f; t)$, $2i\rho$ is an eigenvalue of $\mathbf{B}_n$ with the corresponding right-eigenvector being $\mathbf{f}_n(\rho)$. Moreover, the converse is also true. 
\begin{lem}\label{lem companion matrix expression}
The zeros of the polynomial $S_n(f; t)$ multiplied with $2i$, coincide with the eigenvalues of the matrix $\mathbf{B}_n$ counting algebraic multiplicity. 
\end{lem}
\noindent
Above lemma is shown in terms of more general orthogonal polynomials in \cite{Day Romero}. 
With the scalar normalization, we obtain an alternative expression for $S_n(f; t)$, for $n\geq 0$ such that $a_n\neq 0$, as
\begin{equation}\label{eigenvalue expression of approx. zeros}
S_n(f; t)
=
\frac{a_n}{n!}\det\left(2itI_n-\mathbf{B}_n\right). 
\end{equation}


\subsection{Boundary of zero region of $S_n(f; t)$}
Using the determinant formula for $S_n(f; t)$, we obtain a bound for its zeros. 
\begin{lem}[Gerschgorin's disk theorem, cf. \cite{Yamamoto}]\label{Gerschgorin}
For a matrix $A=(a_{ij})\in M_n(\mathbb{C})$, put
\[
r_j=\sum_{\substack{i=1\\i\neq j}}^n|a_{ij}|,\quad R_j=\{z\in \mathbb{C}\big||z-a_{jj}|\leq r_j\}.
\]
Then, the all eigenvalues of $A$ lie on the union of all circles $ R=\bigcup_{j=1}^n R_j$. If a connected component $\Gamma$ of $R$ is consist of $m$ $R_j$'s, then $\Gamma$ has exactly $m$ eigenvalues counting multiplicity. 
\end{lem}
For $n\geq 0$, put $\{\rho_{n, k}\}_{k=1}^{n}:=Z_n(f)$ to be the zeros of $S_n(f; t)$. And denote the largest zero of $S_n(f; t)$ by $\rho_n^{\max} \in \mathbb{C}$ so that
\begin{equation}
|\rho_n^{\max}|
=
\max_{1\leq k \leq n}\left\{|\rho_{n, k}|\right\}. 
\end{equation}
We continue to consider only $n\geq 0$ such that $a_n(f)\neq 0$. 
\begin{lem}\label{zero-bound}
Retain the same assumption and notations in Lemma \ref{coefficients in MP-ex}. 
Then we have following bound of $\rho_{n}^{\max}$; there is a constant $C>0$ such that
\[
|\rho_{n}^{\max}|\leq Cn.
\]
\end{lem}
\begin{proof}
Recall the definition of $\mathbf{B}_n$ \eqref{companion matrix}; 
a sum of a tridiagonal matrix and a bottom row matrix. 
By Lemma \ref{coefficients in MP-ex}, we can take two constants $C_1, C'_1>0$ such that 
\[
0<C_1 n^{-a-\frac{\nu}{2}} \leq |a_n|\leq C'_1 n^{-a-\frac{\nu}{2}},
\quad
a=\min\left\{|\Re s_-|, |\Re s_+|\right\}>0
\]
for any $n\geq 0$ such that $a_n(f)\neq 0$. Then there exists a constant $C_2>0$ such that
\begin{equation}\label{estimate of frac. coeff.}
\Big|\frac{a_{j-1}}{a_n}\Big| 
\leq
 C_2\left(\frac{j}{n}\right)^{-a}
 \quad
  (0\leq j\leq n-1).
\end{equation}
Put $\mathbf{n}$ to be a diagonal matrix
\[
\mathbf{n}:=\mathrm{diag}[1,2^a, 3^a, \ldots,n^a],
\]
and $\mathbf{B}'_n$ to be a similar matrix to $\mathbf{B}_n$,
\[
\mathbf{B}'_n:=\mathbf{n}^{-1}\mathbf{B}_n\mathbf{n}.
\]
Notice that $\mathbf{B}'_n$ has the same eigenvalues as $\mathbf{B}_n$.
Then, consider to apply the Gerschgorin disk theorem (Lemma \ref{Gerschgorin}) to $\mathbf{B}'_n$. The radius of the disk $R_j$ is 
\[
r_j=(j-1)\left(\frac{j+1}{j-1}\right)^a+(j+\frac12)\left(\frac{j-1}{j+1}\right)^a+n\Big|\frac{a_{j-1}}{a_n}\Big| \left(\frac{j}{n}\right)^a \leq Cn,
\]
with a constant $C\in\mathbb{C}$.
So the region $R=\cup R_j$ is contained in a larger disk
\[
\{z\in \mathbb{C} \big| |z|\leq Cn\}.
\]
Therefore, the zeros $\{\rho_{n, k}\}_{k=1}^{n}$, that are eigenvalues of $\mathbf{B}'_n$ multiplied by $(2i)^{-1}$, are bounded as $|\rho_{n, k}|\leq \frac{C}{2}n$. We get the result replacing the constant $C$.
\end{proof}


%
%
%
\begin{lem}
Retain the same assumption and notations in Lemma \ref{coefficients in MP-ex}. 
Then there exists a positive constant $C'>0$ such that
\begin{equation}
|\rho_n^{\mathrm{max}}|\geq C' n
\end{equation}
for sufficient large $n \geq 0$.
\end{lem}
\begin{proof}
We claim that for $n\geq 0$, the square sum of $\{\rho_{n,k}\}_{k=1}^{n}$ is
\begin{equation}\label{square-zero sum}
\sum_{k=1}^{n}\rho_{n,k}^2= \frac{1}{24} n(n-1)(4n+1)+\frac{1}{2}(n-1)\frac{a_{n-2}}{a_n}.
\end{equation}
The left hand side is equal to $(2i)^{-2}\mathrm{tr}(\mathbf{B}_n^2)$, and it can be computed easily by the definition \eqref{companion matrix}. In fact, it is computed as
\begin{align*}
\mathrm{tr} (\mathbf{B}_n ^2) =-2\sum_{k=1}^{n-1}k(k+\frac12)-2(n-1)\frac{a_{n-2}}{a_n}
=-\frac16 n(n-1)(4n+1)-2(n-1)\frac{a_{n-2}}{a_n}. 
\end{align*}
This is equivalent to \eqref{square-zero sum}. By the inequality $\Big|\frac{a_{j-1}}{a_n}\Big| \leq C_2\left(\frac{j}{n}\right)^{-a}$ in \eqref{estimate of frac. coeff.}, the right hand side of \eqref{square-zero sum} is estimated as $\sim \frac23 n^3 \; (n\rightarrow \infty)$. Therefore we have
\begin{equation*}
n|\rho_n^{\mathrm{max}}|^2
\geq
\Big| \sum_{k=1}^n \rho_{n, k}^2 \Big| \geq C_3 n^3,
\quad
{}^{\exists}C_3 >0,
\end{equation*}
and get the result. 
\end{proof}
\begin{ex}
For the computation of matrix trace in the above proof, we see an example for $n=4$ with general coefficients as
\begin{align*}
\mathrm{tr}\left(
\begin{array}{cccc}
 0 & -b_1 & 0 & 0 \\
 c_1 & 0 & -b_2 & 0 \\
 0 & c_2 & 0 & -b_3 \\
 \frac{a_0}{a_4} & 0 & \frac{a_2}{a_4}+c_3 & 0 \\
\end{array}
\right)^2
=
-2 b_1 c_1-2 b_2 c_2-2 b_3 \left(\frac{a_2}{a_4}+c_3\right).
\end{align*}
\end{ex}

\begin{prop}
There exist positive constants $C, C' >0$ such that
\begin{equation}
C' \leq \frac{|\rho_n^{\max}|}{n} \leq C
\end{equation}
for sufficiently large $n\geq 0$. Here we retain the same assumption and notations in Lemma \ref{coefficients in MP-ex}. 
\end{prop}
The equation \eqref{square-zero sum} also denies the fact that all zeros of $S_n(f; t)$ lie in a neighbour of the imaginary axis because the square sum is positive for large $n$. Now, we concern about how much information we can get from the power sum $\sum_{k=1}^{n}\rho_{n,k}^s$ with $s$ being not only a positive integer.

\medskip 
We continue the study of zeros of $S_n(f; t)$, especially for the completed Riemann zeta function and Dirichlet L-functions. We also consider the MP-expansion of the Selberg zeta function, which is known to satisfy the Riemann hypothesis.  


\vskip 1cm
\noindent
{\bf Acknowledgment:} I am deeply grateful to Prof. Wakayama who provided carefully considered suggestions and valuable comments. I would also like to thank to Hochstenbach, dr. M.E. who gave a fruitful result in the numerical calculation.

%

\begin{flushleft}
\bigskip
Hiroto Inoue\par
Graduate School of Mathematics, \par
Kyushu University \par
744 Motooka, Nishi-ku, Fukuoka 819-0395 JAPAN \par
\texttt{hi-inoue@math.kyushu-u.ac.jp}
\par
\end{flushleft}

\end{document}